\documentclass[12pt,reqno]{amsart}
\usepackage{amssymb,amsthm,amsmath}
\usepackage{floatrow}
\usepackage{mathabx}
\usepackage{float}
\usepackage{mathtools}
\usepackage{subfigure}
\usepackage{geometry}
\usepackage{lipsum}
\usepackage{caption} 
\usepackage{hyperref}
\usepackage{tikz}
\usepackage{array}
\usepackage{enumitem}
\usepackage{multirow}

\usepackage{graphicx}

\divide
\oddsidemargin by 2 
\date{}
\makeatletter
\def\ps@pprintTitle{
  \let\@evenfoot\@oddfoot
}
\makeatother

\newtheorem{theorem}{Theorem}

\newtheorem{remark}{Remark}
\newtheorem{lemma}{Lemma}

\newtheorem{question}{Question}

\makeatletter
\newcommand{\namelabel}[1]{%
  \phantomsection
  \renewcommand{\@currentlabel}{#1}
  \label{#1}
}
\makeatother

\begin{document}
\title{Simple lift of non-simple closed curves}

 \author{Deblina Das}
 \address{Department of Mathematics, Indian Institute of Technology Palakkad}
 \email{212114002@smail.iitpkd.ac.in}
 \author{Arpan Kabiraj}
 \address{Department of Mathematics, Indian Institute of Technology Palakkad}
\email{arpaninto@iitpkd.ac.in}
\begin{abstract}
Given a compact, oriented surface $S$ of finite genus and finitely many boundary components, we construct explicit examples of finite covers $\tilde{S}$ of $S$ and non-simple closed curves $\gamma$ on $S$ which lift to simple closed curves on $\tilde{S}$. In particular, given any positive integer $n\geq 2$, we construct explicit non-simple closed curves on $S$ that have simple lifts to a degree $n$ cover of $S$.
\end{abstract}
\maketitle
\section{Introduction}
Let $S$ be a connected, compact, oriented surface of finite genus and finitely many boundary components. We call such a surface a \emph{finite-type surface}. 
The \emph{self-intersection number} $i(\gamma,\gamma)$ of a free homotopy class of a closed curve $\gamma$ on $S$ is the minimal number of self-intersection points of the representative curves in the class of $\gamma$. Abusing notation, we identify a free homotopy class by its representatives. A closed curve $\gamma$ is said to be \emph{simple} if $i(\gamma, \gamma)=0$ and \emph{non-simple} if $i(\gamma,\gamma)\neq 0$. It is called \emph{primitive} if it is not the positive power of another curve.  A closed curve $\gamma$ is called \emph{essential} if it is not homotopic to a point or to any boundary component of $S$.  We call two finite covering spaces $S_1$ and $S_2$ of $S$ \emph{topologically equivalent} if $S_1$ and $S_2$ are homeomorphic as topological spaces (not necessarily as covering spaces). 

It is a celebrated theorem of Peter Scott \cite{scott1978subgroups}, \cite{scott1985correction} that each  non-simple closed
curve $\gamma$ lifts to a simple closed curve on some finite-sheeted cover (i.e. $\gamma$ ‘lifts simply'). In view of this result, it is natural to ask the following question:
\begin{question}\label{Q1}
    Given any surface $S$ and any finite cover $\tilde{S}$ of $S$, does there exist a non-simple primitive closed curve $\gamma$ on $S$ which admits a simple lift to $\tilde{S}$?
\end{question}

Question \ref{Q1} has an affirmative answer. This follows directly from  \cite[Theorem 3]{MR4887777}. It can also be deduced from \cite{MR2805069} via the following argument. 
Suppose the above result is false. Then the map from the curve complex of $\tilde{S}$ to the curve complex of $S$ induced by the covering map is coarsely defined and hence is a quasi-isometry.   
Now it follows from \cite[Theorem 1.2]{MR2805069}  that the covering map is a homeomorphism.  


However, both the proofs mentioned above are existential and rely on the geometry of the curve complex and  Teichm\"{u}ller space. 

This article is motivated by the search for an elementary, constructive, and purely topological proof of this result. 
The main result of this article is the following theorem.
\begin{theorem}\label{Main theorem}
    Let $S$ be a compact, connected, oriented, finite-type surface with Euler characteristic $\chi(\Sigma)$. For any positive integer $n\geq 2$, and for any surface $\tilde{S}$ with $\chi(\tilde\Sigma)=n\chi(\Sigma)$  we can explicitly construct a covering map $p:\tilde{S}\to S$ and a non-simple primitive closed curve $\gamma$ on $S$ such that $\gamma$ lifts simply to $\tilde{S}$. 
\end{theorem}
Although this does not answer Question \ref{Q1} in its full generality, the result provides a partial answer. Indeed, instead of proving the statement for a prescribed finite cover, we construct a finite cover that is topologically equivalent to the prescribed one. Our approach does not rely on the geometry of the curve complex or Teichmüller space, but instead gives explicit topological constructions of both the finite-sheeted covers and the non-simple curves that lift simply. We hope that this approach may lead to an affirmative answer to Question \ref{Q1} in full generality

 We denote the minimum degree of covering to which the curve $\alpha$ lifts simply  by $d(\alpha)$. Recently, there has been a considerable amount of research to quantify Scott’s result by obtaining bounds on $d(\alpha)$ in terms of topological and geometric properties of $\alpha$. See, \cite{patel2014theorem}, \cite{aougab2016building}, \cite{gaster2016lifting},\cite{MR3893306},\cite{MR4535842}, \cite{MR4056691} for details.

\subsection{Outline of the proof}
Our approach relies on the construction of finite-sheeted coverings of surfaces with boundaries given in \cite{massey1974finite}. We recall the relationship among the Euler characteristic $\chi$, the genus $g$ and the number of boundary components $k$ of the compact surface $S$:\begin{equation}\label{equation 1} 
    \chi=2-2g-k.
\end{equation} 
 A given $n$-sheeted covering $\tilde{S}$ of $S$ must satisfy the following conditions. If Euler characteristics of $S$ and $\tilde{S}$ be $\chi$ and $\tilde{\chi}$ respectively, then \begin{equation}\label{equation 2}
   \tilde{\chi}=n\chi.
\end{equation}
If we denote boundary components of $S$ and $\tilde{S}$ by $k$ and $\tilde{k}$ respectively, then
\begin{equation}\label{equation 3}
    k\leq \tilde{k}\leq nk.
\end{equation}
We recall that any covering of an oriented surface is oriented and any finite-sheeted covering of a compact surface is  compact. Therefore (up to topological equivalence) finite-sheeted covers are completely determined by any two of the following three numbers $(a)$ genus, $(b)$ number of boundary components and $(c)$ Euler characteristics. The main theorem in \cite{massey1974finite} gives a construction of $n$-sheeted coverings $\tilde{S_1}$ of $S$ satisfying the conditions given in the equation (\ref{equation 1}), equation (\ref{equation 2}) and equation (\ref{equation 3}) above by fixing Euler characteristics and for all possible boundary components of the coverings or by fixing Euler characteristics and for all possible genus of the coverings. Now,  given any $n$-sheeted covering $\tilde S$ of $S$, it  must satisfy the equation (\ref{equation 1}), equation (\ref{equation 2}) and equation (\ref{equation 3}). Therefore, $\tilde{S}$ and one of the coverings $\tilde{S_1}$ both have the same Euler characteristics and the same number of boundary components. Due to the classification theorem, up to homeomorphism, a compact surface with $k\ (\geq 1)$ boundary components is a surface of genus $g\ (\geq 0)$ with $k$ boundaries. Hence $\tilde{S}$ and one of the coverings $\tilde{S_1}$ are homeomorphic, i.e., topologically equivalent. Therefore, without loss of generality, we prove the Theorem \ref{Main theorem} for surfaces with boundaries considering only the particular $n$-sheeted coverings $\tilde{S_1}$ constructed in \cite{massey1974finite}. 

The proof of Theorem \ref{Main theorem} proceeds by analyzing the possible cases according to the genus and the number of boundary components as outlined in Section \ref{Pair of pants} - Section \ref{Surface of genus greater equal to 1  boundary components greater equal to 3}. Finally, in Section \ref{Closed orientable surfaces}, we discuss the proof for compact surfaces without boundaries. Throughout this article, we denote the surfaces of genus $g$ with $k$ boundary components as $S_{g,k}$ and consider only those surfaces with negative Euler characteristics unless otherwise specified.

\section*{Acknowledgements}
We would like to thank Macarena Arenas for pointing out the reference \cite{MR4056691}. The second author acknowledges the support from DST SERB Grant No.: SRG/2022/001210 and  DST SERB Grant No.: MTR/2022/000327.

\section{Surfaces with non-empty boundary }
\subsection{Pair of pants.}\label{Pair of pants}
Let $P$ be a pair of pants, i.e., a surface of genus $0$ and $3$ boundary components. An $n$-sheeted covering of $P$ is denoted by $P^{[n]}$ where $n\in \mathbb{N}_{\geq 2}$.
We identify the fundamental group of $P$, $\pi_1(P,v)$ as a rank $2$ free group $F_2$ with generators $a$ and $b$. Recall that a \emph{ribbon graph} also called \emph{fat graph} is a graph equipped with a cyclic ordering on the half edges incident to each vertex.
We choose orientations for $a$ and $b$ and draw the pair of pants $P$ as a directed ribbon graph with one vertex $v$ and two directed edges $a$ and $b$ as in Figure \ref{Pair of pants $P$ with generators}. 
\begin{figure}[ht]
    \centering
     \includegraphics[width=6cm, height=4cm]{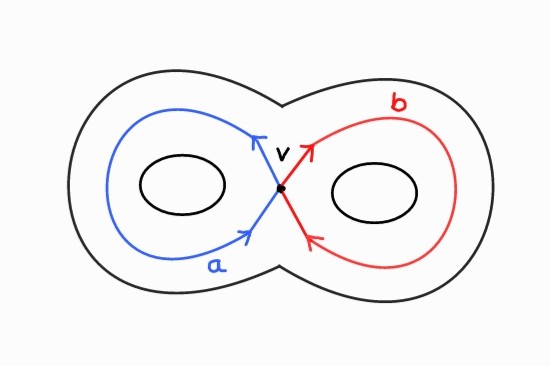}
    \caption{Pair of pants $P$ with generators $a$ and $b$.}
    \label{Pair of pants $P$ with generators}
\end{figure}

  If the number of boundary components of $P^{[n]}$ is $\tilde{k}$, then the following condition is satisfied (see \cite[Section 1]{massey1974finite}): \[3\leq \tilde{k}\leq n+2, \qquad \tilde{k}\equiv n\pmod 2.\]
 Also, all values of $\tilde{k}$ satisfying this condition are realized by some $n$-sheeted covering of $P$. A construction is given in \cite[Section 2]{massey1974finite} which we recall below. 
 
 For each integer $m$ depending on $\tilde{k}$ and $0\leq m \leq n-2$,  an explicit $n$-sheeted covering of $P$ is given by $p: P^{[n]}_m \to P$ as follows. We draw $P_m^{[n]}$ as directed ribbon graph with $n$ vertices and $2n$ directed edges, $n$ with $a$ labels and $n$ with $b$ labels. These coverings satisfy the following conditions:
 
 \begin{itemize}
     \item The inverse image $p^{-1}(b) $ is connected.
     \item The inverse image $p^{-1}(a) $ has $m+1$ components. Among them, $m$ components map homeomorphically onto $a$ by $p$ and the remaining one component is $(n-m)$-sheeted covering of circle $a$. Moreover, the order in which the vertices  $P^{[n]}_m$ occur on this $(n-m)$-sheeted covering matches the order in which these vertices occur around the preimage $p^{-1}(b)$.
 \end{itemize}  
 We fix the order of the vertices on $P^{[n]}_m$ by $v_0,v_1,v_2,\ldots,v_m,v_{m+1},\ldots,v_{n-1}$ which are preimages of $v$ in this covering space such that $v_1,v_2,\ldots, v_m$ (provided $m\geq 1$) are the vertices where directed edge of label $a$ occurs as homeomorphic preimage of $a$.
 \begin{lemma}\label{lemma}
     With the above notation, for this fixed ordering of the vertices, there exists one directed edge with label $a$ on $P^{[n]}_m$ starting at $v_{m+1}$ and ending at $v_{0}$ where $m\geq 0$.
 \end{lemma}
 \begin{proof}
 By the construction of the covering, there is an edge labeled $a$ that starts at $v_{m+1}$ and ends at $v_{k}$ for some $k\geq 0$. We observe that $k\neq 1,2, \ldots ,m,m+1$, since \{$v_1,v_2,\ldots, v_m$\} is the set of vertices where a loop labeled $a$ occurs as homeomorphic preimage of $a$. If $k\in \{m+2,m+3, \ldots,n-1\}$ then the order in which the vertices of $P^{[n]}_m$ occur on this $(n-m)$-sheeted covering does not match with the order of the vertices around the preimage $p^{-1}(b)$. So the second property of the covering construction is violated. Therefore, $k=0$. This completes the proof. \end{proof}
                      
 Let, $\gamma^k$ be the closed curve on $P$ given by the class $ab^k$ where $k\in \mathbb{N}$. We observe that $i(\gamma^k, \gamma^k)=k$ for all $ k \in \mathbb{N}$ (see, \cite{gaster2016lifting}) and therefore $\{\gamma^k\}$ is a collection of non-simple closed curves on $P$ (e.g., Figure \ref{Representative of gamma_3 gamma_1}). A lift of $\gamma^{k}$ is given by one directed edge labeled $a$ followed by $k$ consecutive directed edges labeled $b$.
 
\begin{figure}[ht]
    \centering
   \subfigure[A minimal position representative of $\gamma^1$.]{\includegraphics[width=6cm, height=4cm]{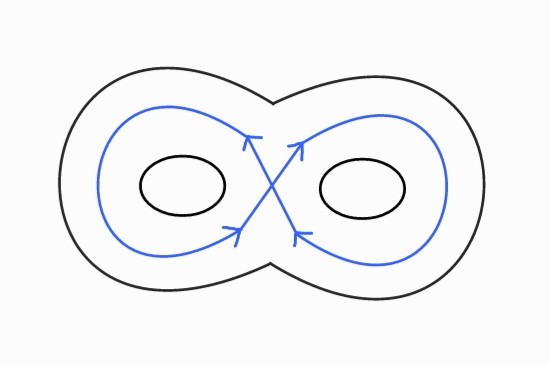}} \hfill 
   \subfigure[A minimal position representative of $\gamma^3$.]{\includegraphics[width=6cm, height=4cm]{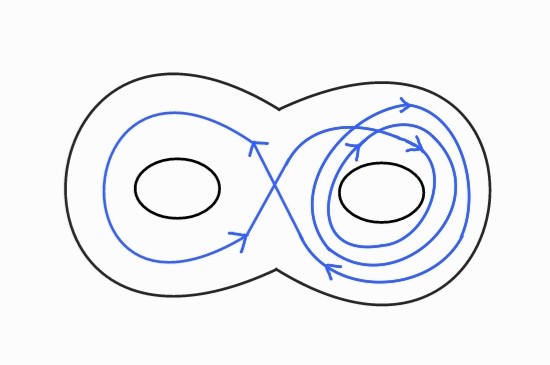}} 
    \caption{Non-simple closed curves $\gamma^k \equiv ab^k$ on $P$.}
  \label{Representative of gamma_3 gamma_1} 
\end{figure}

For each integer $m$ such that $0\leq m \leq n-2$, we choose non-simple closed curves $\gamma^{m+1}$ on $P$. Due to the construction of the covering space and Lemma \ref{lemma} we observe that a lift of $\gamma^{m+1}$ is given by one labeled $a$ edge which starts at $v_{m+1}$ and ends $v_{0}$; then followed by $m+1$ consecutive labeled $b$ edges joining $v_0$ to $v_1$, $v_1$ to $v_2$, $\ldots$, $v_m$ to $v_{m+1}$. Hence, $\gamma^{m+1}$ lifts simply on $P_m^{[n]}$ by $p$.
All the $6$-sheeted covering spaces of $P$ arising from this construction and simple lifts of $\gamma^{m+1}$ are depicted in Figure \ref{Scheme of covering spaces of P for n=6 }.

\begin{figure}[h]
    \centering
\subfigure[Covering space of $P$ in case of $m=0$ and a simple lift of $\gamma^1\equiv ab^1$ in this covering.] 
{{\includegraphics[width=6cm, height=6cm]{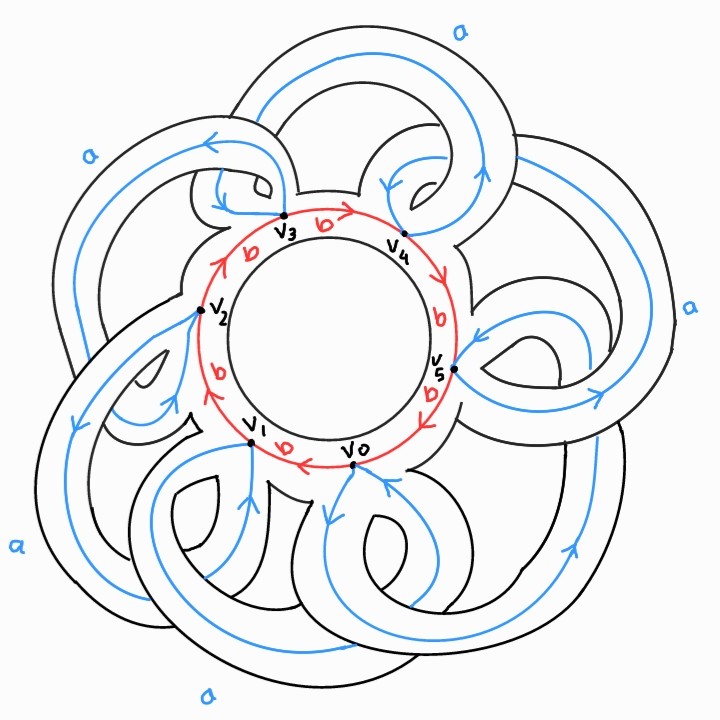}}\hspace{1cm}
{\includegraphics[width=6cm, height=6cm]{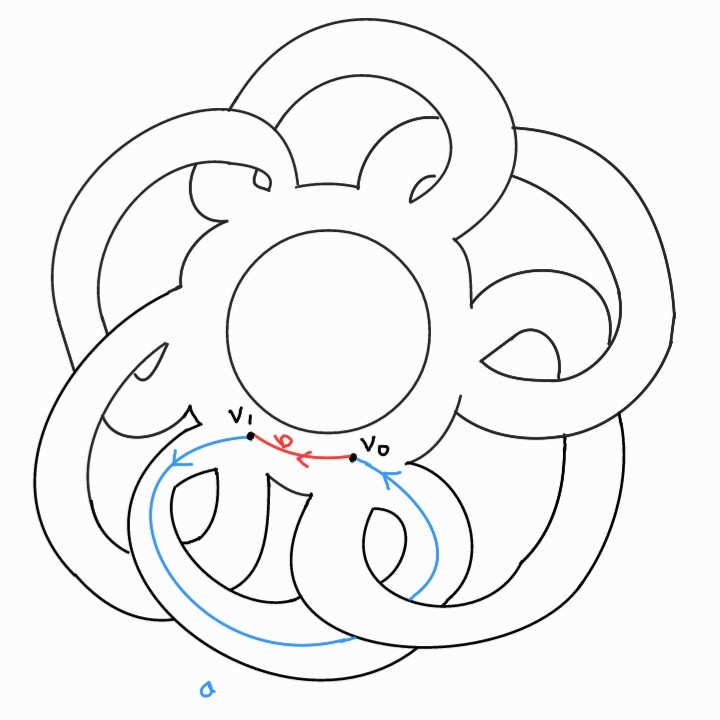}}}
\hfill
\end{figure}
\begin{figure}[h]
    \centering
     \subfigure[Covering space of $P$ in case of $m=1$ and a simple lift of $\gamma^2\equiv ab^2$ in this covering.]{{\includegraphics[width=6cm, height=6cm]{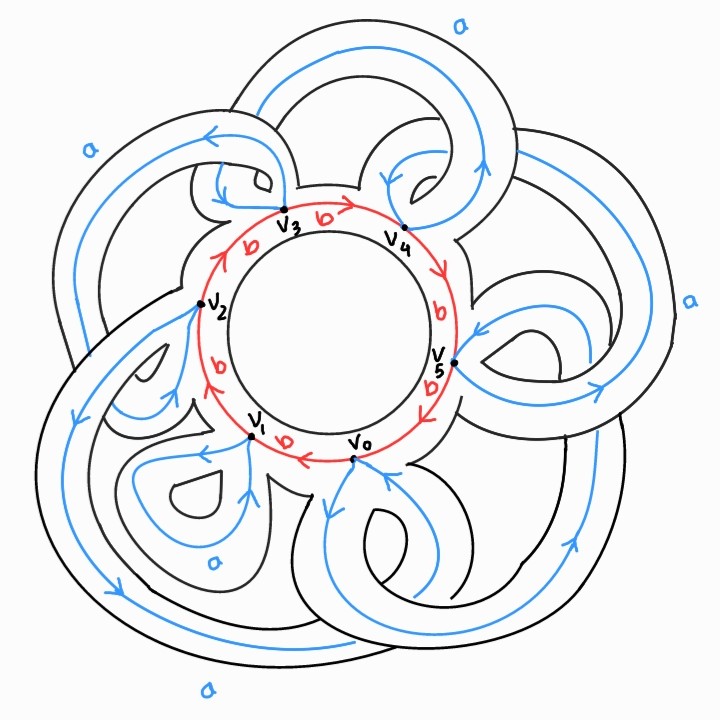}} \hspace{1cm}
   {\includegraphics[width=6cm, height=6cm]{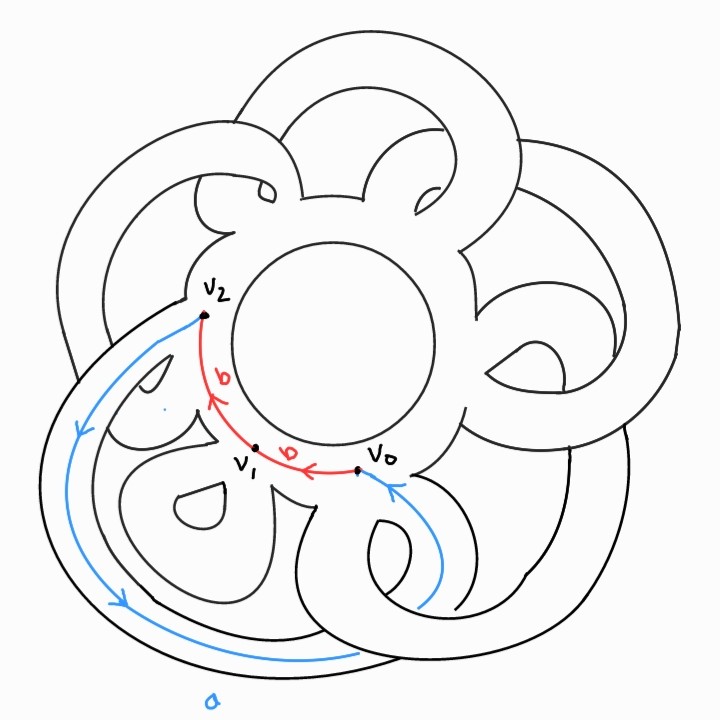}}} \hfill
   \end{figure} 
   \begin{figure}[H]
   \centering
 \subfigure[Covering space of $P$ in case of $m=2$ and a simple lift of $\gamma^3\equiv ab^3$ in this covering.]{{\includegraphics[width=6cm, height=6cm]{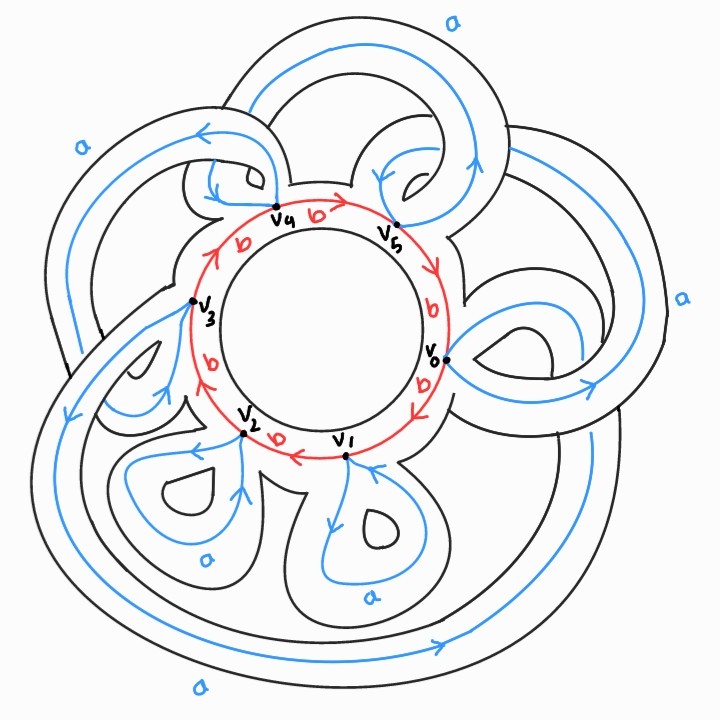}} \hspace{1cm}
   {\includegraphics[width=6cm, height=6cm]{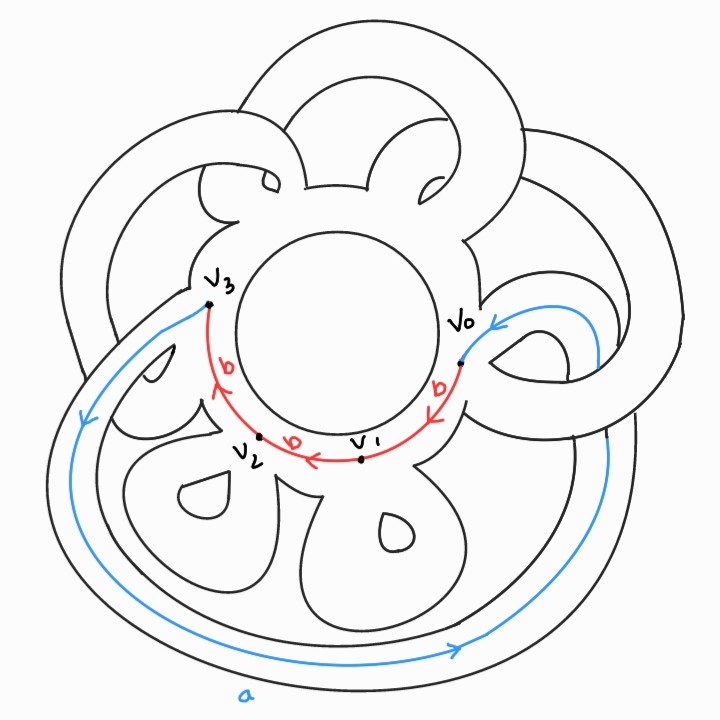}}}\hfill
   \end{figure}
 \begin{figure}[ht]
   \centering
    \subfigure[Covering space of $P$ in case of $m=3$ and a simple lift of $\gamma^4\equiv ab^4$ in this covering.]{{\includegraphics[width=6cm, height=6cm]{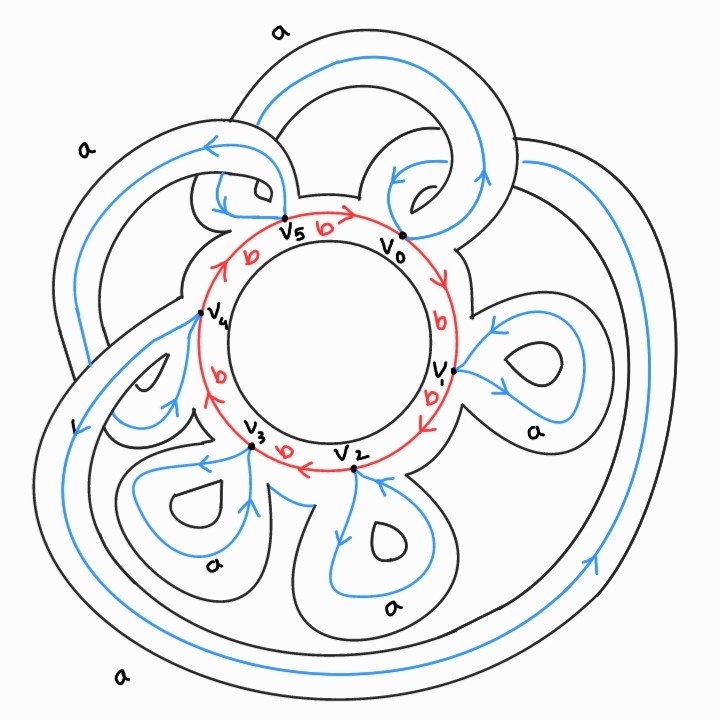}} \hspace{1cm}
   {\includegraphics[width=6cm, height=6cm]{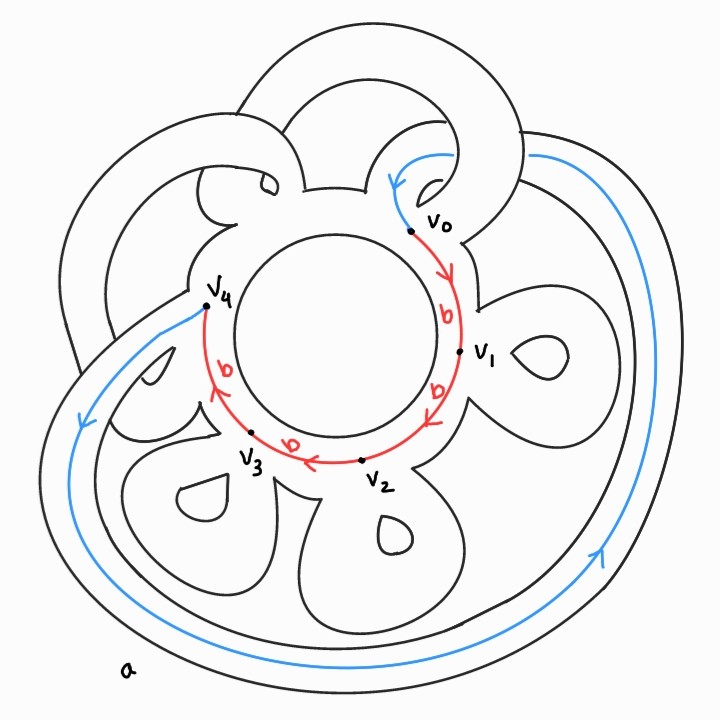}}}\hfill   
\end{figure}
\begin{figure}[H]
    \centering
\subfigure[Covering space of $P$ in case of $m=4$ and a simple lift of $\gamma^5\equiv ab^5$ in this covering.] 
{{\includegraphics[width=6cm, height=6cm]{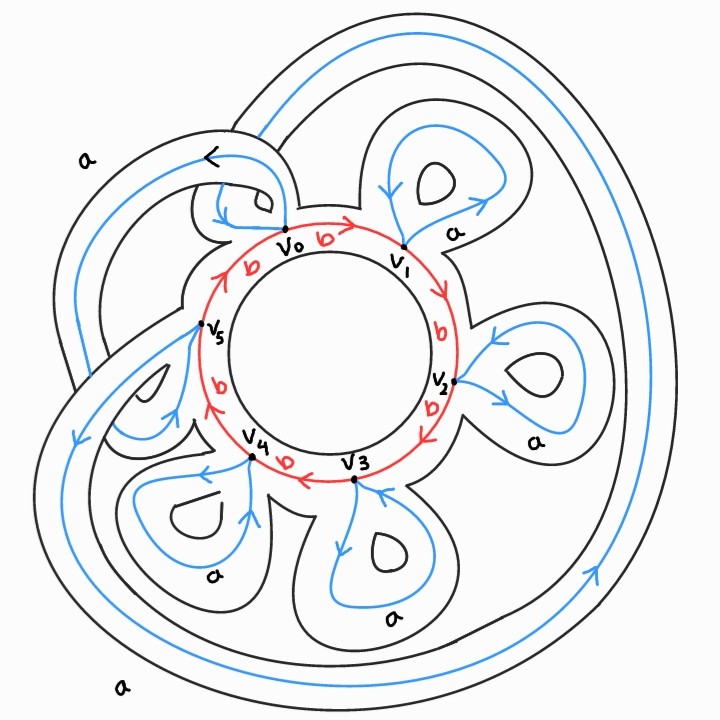}}\hspace{1cm}
{\includegraphics[width=6cm, height=6cm]{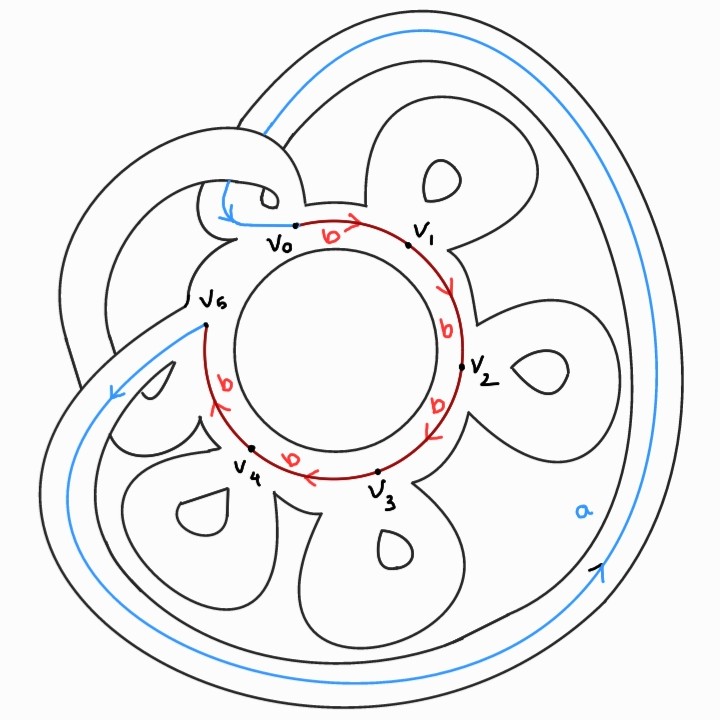}}}
\caption{Scheme of $6$-sheeted covering spaces of $P$ and simple lifts of $\gamma^{m+1}$ where $0\leq m\leq 4$.}
\label{Scheme of covering spaces of P for n=6 }
\end{figure}

\begin{remark}
   For given two topologically non-equivalent $n$-sheeted covers $P_1$ and $P_2$ of $P$, we can always construct a non-simple closed curve on $P$ which lift to $P_1'$ simply but not to $P_2'$ where $P_1'$ and $P_2'$ are  $n$-sheeted covers of $P$ such that $P_1'$ is topologically equivalent to $P_1$; $P_2'$ is topologically equivalent to $P_2$.
\end{remark}

\subsection{Surface of genus $g=0$ and boundary components $k\geq 3$.}\label{Surface of genus g=0 and boundary components k greater equal to 3}
We denote a surface of genus $g=0$ and $k$ boundary components by $S_{0,k}$ and its $n$-sheeted covering by $S_{0,k}^{[n]}$. In this case the number of boundary components  $\tilde{k}$ of an $n$-sheeted covering space satisfies the conditions (see \cite[Section 4]{massey1974finite}):
\[k\leq \tilde{k} \leq (k-2)n+2, \qquad \tilde{k} \equiv nk \pmod2 .\]
Also, any integer $\tilde{k}$ satisfying these conditions can be realized by some $n$-sheeted covering of $S_{0,k}$ as described in \cite[Section 4]{massey1974finite}, which we recall below.
We consider two cases to construct non-simple closed curves that lift simply on $S_{0,k}^{[n]}$. Note that the construction of covering $S_{0,k}^{[n]}$ is depending on $\tilde{k}$, but for simplicity we avoid include it in the notation of $S_{0,k}^{[n]}$.

\textit{Case 1:}\phantomsection\label{Case 1,S_0,1} Suppose the number of boundary components of $S_{0,k}$ is even. Let $A_1,A_2,\ldots ,A_k$ denote the boundary components of $S_{0,k}$. We make disjoint cuts $c_1,c_2,\ldots, c_{\frac{k}{2}}$ on $S_{0,k}$ such that $c_1$ joins $A_1$ and $A_2$, $c_2$ joins $A_3$ and $A_4$ and so on. The method is shown in Figure \ref{Surface S_{0,k}} in the case of $k=6$.
\begin{figure}[ht]
    \centering
     \includegraphics[width=7cm, height=5cm]{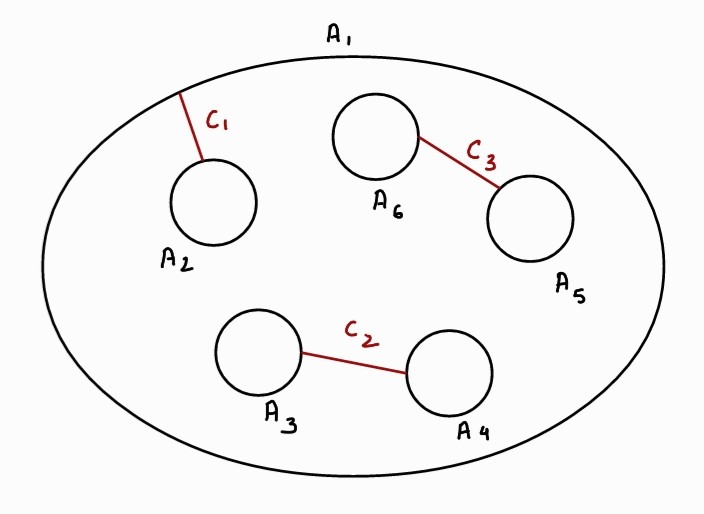}
    \caption{Surface $S_{0,6}$ with boundary components joined by cuts.}
    \label{Surface S_{0,k}}
\end{figure}

To construct a $n$-sheeted covering space, we take such $n$ copies of $S_{0,k}$ with indicated cuts and paste the edges of the cuts in different copies of $S_{0,k}$: we paste along the $c_1$ cut so that $p^{-1}(A_1)$ and $p^{-1}(A_2)$ are connected and ensure that $S_{0,k}^{[n]}$ is connected. The remaining cuts may be pasted together so as to get either one of the extreme conditions:
\begin{enumerate}
    \item The inverse images $p^{-1}(A_3)$, $p^{-1}(A_4)$, $\ldots$, $p^{-1}(A_k)$ each have $n $ components. In this case $\tilde{k}=(k-2)n+2$.
    \item The inverse images $p^{-1}(A_3)$, $p^{-1}(A_4)$, $\ldots$, $p^{-1}(A_k)$ are all connected. In this case $\tilde{k}=k$.

\end{enumerate}
Any intermediate values of $\tilde{k}$ between these extremes can be achieved by pasting cuts differently.
Let $p:S_{0,k}^{[n]}\to S_{0,k}$ be the $n$-sheeted coverings obtained by this construction.
We recall, the fundamental group of $S_{0,k}$ is a free group ${F}_{k-1}$ with $k-1$ generators. The generators around $A_2, A_3,\ldots A_k$ are denoted by $a_2, a_3, \ldots, a_k$  respectively. Without loss of generality, we choose the orientation of all generators in an anti-clockwise direction. A set of chosen generators with orientation for $S_{0,6}$ are shown in Figure \ref{tau in terms of generators}(a). Let, $\tau^j$ be the closed curve on $S_{0,k}$ given by the homotopy class $a_2a_3\ldots a_ka_2^{j}$ where $j\in \mathbb{N}$. 
We observe that $i(\tau^j, \tau^j)=j$ (using bigon criteria) for all $j\in \mathbb{N}$ and hence $\{\tau^j\}$ is a collection of non-simple closed curves on $S_{0,k}$.

\begin{figure}[ht]
    \centering
   \subfigure[Generators of $S_{0,6}$ with chosen orientation.]{\includegraphics[width=7cm, height=5cm]{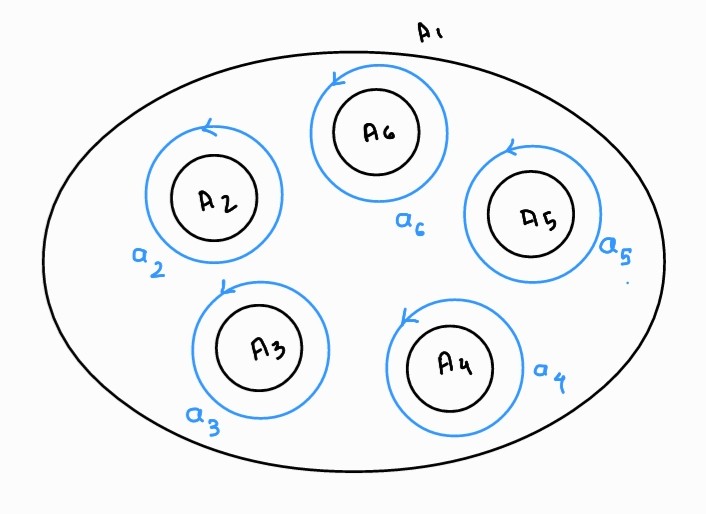}} \hfill 
   \subfigure[A minimal position representative of $\tau^2$ on $S_{0,6}$.]{\includegraphics[width=7cm, height=5cm]{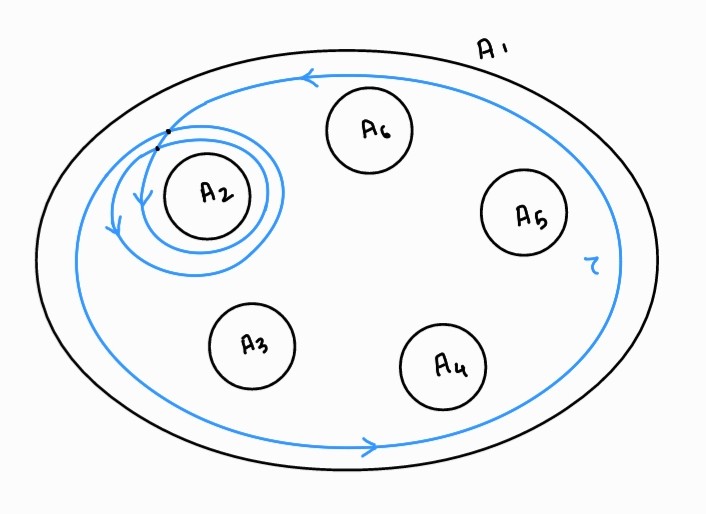}} 
    \caption{}
  \label{tau in terms of generators} 
\end{figure}
\begin{figure}[ht]
    \centering
     \includegraphics[width=8cm, height=6cm]{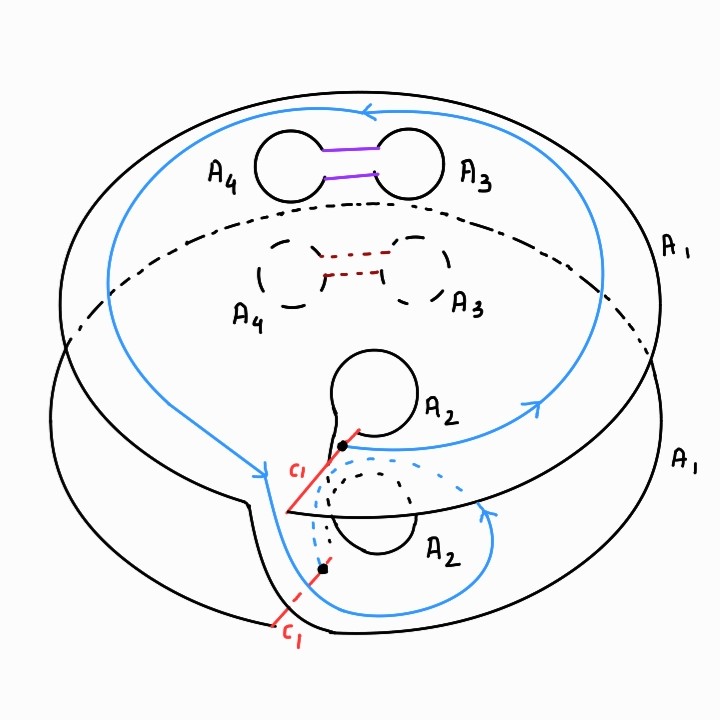}
    \caption{Lift of non simple closed curve $\tau^1 \equiv a_2a_3a_4a_2$ on $S_{0,4}^{[2]}$.}
    \label{Lift of tau in 2 sheeted covering of S_0,4}
\end{figure}

  In this case, we choose $\tau^{n-1}$ on $S_{0,k}$ for all possibilities of $\tilde{k}$ (e.g. Figure \ref{tau in terms of generators}$(b)$ shows the chosen non-simple closed curve in the case of $n=3$ on $S_{0,6}$). It follows directly from the construction of $p$ that these curves lift simply on $S_{0,k}^{[n]}$. For example, in Figure \ref{Lift of tau in 2 sheeted covering of S_0,4}, a simple lift of $\tau^1$ is shown on $2$-sheeted covering $S_{0,4}^{[2]}$. Here the red cut lines $c_1$ are identified where two points `$\bullet$' are identified on both sheets. No matter how the rest of the cuts are identified maintaining the scheme described, the indicated blue curve gives a simple lift of $\tau^1$ on $S_{0,4}^{[2]}$. This concludes Case 1. 
  
  Before going to the next case, we state the following well known lemma.
  \begin{lemma}\label{Essential closed curve}
    Any lift of an essential closed curve on a surface $S$ is again essential.
\end{lemma}
\begin{proof}[Proof]
    Let $p: \tilde{S} \to S $ be a covering projection and $\tilde{\gamma}$ on $\tilde{S}$ is a lift of an essential closed curve $\gamma$ on $S$. If possible, $\tilde{\gamma}$ is non-essential. Then there is a boundary component of $\tilde{S}$, say $\sigma$ such that
\[ \tilde{\gamma} \sim \sigma \implies p^{-1}(\gamma )\sim \sigma \implies p(p^{-1}(\gamma))\sim p \circ \sigma  \implies \gamma \sim p\circ \sigma.\]
Since $p$ is a local homeomorphism, $p\circ\sigma$ is homotopic to a boundary component of $S$, i.e., $\gamma$ is homotopic to a boundary of $S$. This gives a contradiction to the fact that $\gamma$ is essential.
\end{proof}
\textit{Case 2:}\phantomsection\label{Case 2, S_0,k} Suppose the number of boundary components of $S_{0,k}$ is odd where $k>3$ (case: $k=3$ is discussed in Section \ref{Pair of pants}). Then $S_{0,k}$ can be represented as the union of two surfaces $N_1$ and $N_2$ where $N_1$ is a surface of genus $0$ and $k-1$ boundary components whereas $N_2$ is a surface of genus $0$ and $3$ boundary components. Also, $N_1 \cap N_2$ is a circle which is a boundary component of both $N_1$ and $N_2$. For example, a representation of  $S_{0,7}$ is depicted in Figure \ref{Union of N_1 and N_2}.

\begin{figure}[ht]
    \centering
     \includegraphics[width=7cm, height=5cm]{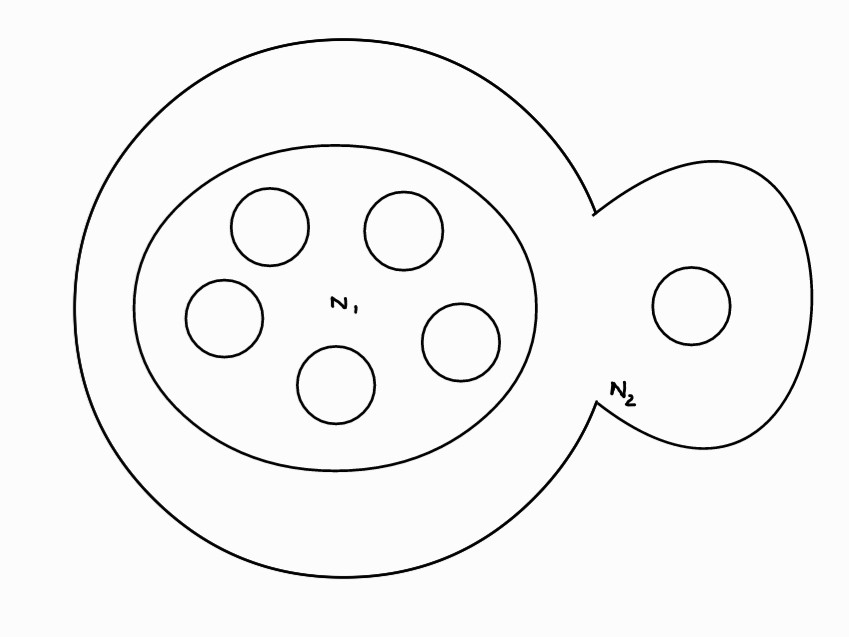}
    \caption{Decomposition of $S_{0,7}$ as union of two surfaces $N_1$ and $N_2$.}
    \label{Union of N_1 and N_2}
\end{figure}
 
 We fix a representation $\{N_1,N_2\}$ of $S_{0,k}$. Now we construct $n$-sheeted covering $p:S_{0,k}^{[n]}\to S_{0,k}$ by following the construction given in \cite[Section 4]{massey1974finite}. At first, we construct $n$-sheeted covering $p_1:\tilde{N_1}\to N_1$ by the method described in \textit{Case 1} (Section \ref{Case 1,S_0,1}) such that $p_1^{-1}(N_1 \cap N_2)$ is connected and construct $n$-sheeted covering $p_2: \tilde{N_2}\to N_2$ by the method described in Section \ref{Pair of pants} such that $p_2^{-1}(N_1 \cap N_2)$ is connected. Then $ S_{0,k}^{[n]}$ is quotient space of $\tilde{N_1}\cup \tilde{N_2}$ identifying $p_2^{-1}(N_1 \cap N_2)$ and $p_1^{-1}(N_1 \cap N_2)$ by appropriate homeomorphism. Let $q:p_1^{-1}(N_1)\cup p_2^{-1}(N_2) \to p_1^{-1}(N_1)\cup p_2^{-1}(N_2)/\sim $ be the required quotient map and $j: p_1^{-1}(N_1)\to p_1^{-1}(N_1)\cup p_2^{-1}(N_2)$ be the inclusion map. Now by \textit{Case 1}, we can construct suitable non-simple closed curves $\tau$ on $N_1$ which lift simply to $\tilde{\tau}$ by $p_1$ on $\tilde{N_1}$. Hence the curves $i \circ \tau$ on $N_1 \cup N_2$ lifts to $q \circ j \circ \tilde{\tau}$ on $S_{0,k}^{[n]}$ by $p$ where $i:N_1 \to N_1 \cup N_2$ is an inclusion map. We observe that the chosen curve $\tau$ on $N_1$ is essential, therefore so is $\tilde{\tau}$ by Lemma \ref{Essential closed curve}. Hence image of $\tilde{\tau}$ under $q$ maps as an identity function no matter how $p_2^{-1}(N_1 \cap N_2)$ and $p_1^{-1}(N_1 \cap N_2)$  are glued together to get $ p_1^{-1}(N_1)\cup p_2^{-1}(N_2)/\sim $. Also $\tilde{\tau}$ and $j$ are injective. Therefore, $q \circ j \circ \tilde{\tau}$ on $S_{0,k}^{[n]}$ is simple lift of $i \circ \tau$ by $p$. This completes the proof for surfaces of genus $0$ with $k$ boundary components for $k\geq 3$.

\subsection{Surface of genus $g=1$ and connected boundary}\label{Section Surface of genus g=1 and connected boundary}
We denote the surface with genus, $g=1$ and boundary component, $k=1$ by $S_{1,1}$. Such a surface is a deformation retract to $S^1 \vee S^1$ so the fundamental group is a free group with two generators $a$ and $b$ as depicted in Figure \ref{Generators of S in case g=1, k=1}.

\begin{figure}[ht]
    \centering
      \includegraphics[width=4cm, height=5cm]{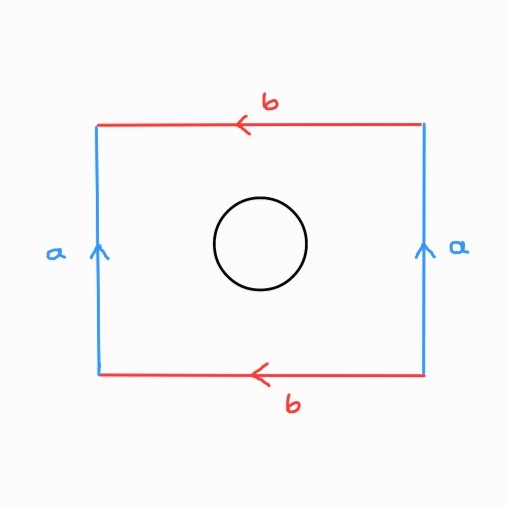}
  \includegraphics[width=5cm, height=5cm]{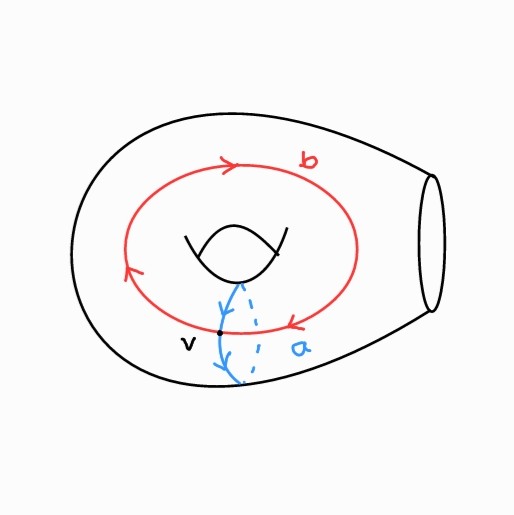}
  \includegraphics[width=5cm, height=5cm]{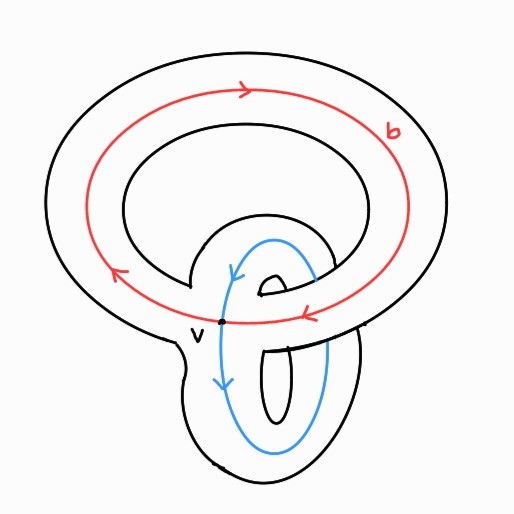}
    \caption{Generators of $S_{1,1}$.}
  \label{Generators of S in case g=1, k=1} 
\end{figure}

We choose orientations for $a$ and $b$ so that $S_{1,1}$ becomes a directed ribbon graph with one vertex and $2$ directed edges. The number of boundary components, $\tilde{k}$ of $n$-sheeted covering $S_{1,1}^{[n]}$ satisfy the inequality (see \cite[Section 1]{massey1974finite}):
\[1\leq \tilde{k}\leq n, \qquad \tilde{k}\equiv n\pmod 2.\]

Also for any integer $q$ satisfying this condition, we construct a specific $n$-sheeted covering $p: ({S_{1,1}^{[n]})_q} \to S_{1,1}$ described in \cite[Section 2]{massey1974finite} which we recall briefly. We  draw $({S_{1,1}^{[n]})_q}$ as a directed ribbon graph with $n$ vertices and $2n$ directed edges, $n$ with $a$ labels and $n$ with $b$ labels and these coverings satisfy the following conditions:
\begin{itemize}
    \item The inverse image $p^{-1}(b)$ is connected.
    \item The inverse image $p^{-1}(a)$ has $q+\frac{1}{2}(n-q)$ components; $q$ of them are mapped homeomorphically onto $a$ by $p$ and each of the remaining $\frac{1}{2}(n-q)$ components consists of a $2$-sheeted covering of $a$. Also the intersection with $p^{-1}(b)$ of such a $2$-sheeted covering consists of two successive vertices of the circle $p^{-1}(b)$.
\end{itemize}
Let, $\eta^k$ be the closed curve on $S_{1,1}$ given by the homotopy class $a^2b^{k}$ where $k\in \mathbb{N}_{\geq 2}$. We observe that $i(a^2b^k,a^2b^k)=k-1$ (see, \cite{Chas2010}, \cite{chemotti2004intersection}) for all $k\in \mathbb{N}$ and therefore $\{\eta^k\}$ is a collection of non-simple closed curves on $S_{1,1}$ for all $k\geq 2$ (e.g. Figure \ref{Representative of eta_2 eta_1}). Since $q\geq 1$, for any $q$ satisfying the condition, $p^{-1}(a)$ has at least one component mapped homeomorphically onto $a$ by the construction of the covering space. Therefore, a lift of $\eta^n$ is given by two labeled $a$ edges followed by $n$ consecutive labeled $b$ edges (see Figure \ref{3 sheeted Covering space of S 1 1}). Consequently, this collection of non-simple closed curves lifts simply on $n$-sheeted covering $S_{1,1}^{[n]}$.
\begin{figure}[ht]
    \centering
   \subfigure[A minimal position representative of $\eta^2\equiv a^2b^2$.]{\includegraphics[width=6cm, height=4cm]{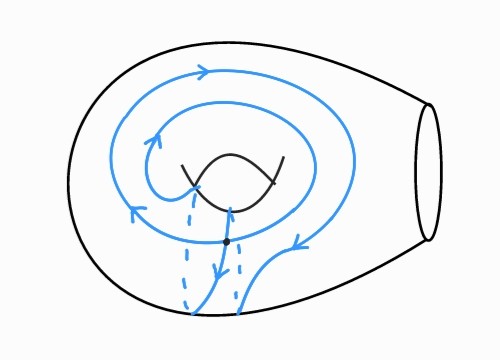}} \hfill 
   \subfigure[A minimal position representative of $\eta^3 \equiv a^2b^3$.]{\includegraphics[width=6cm, height=4cm]{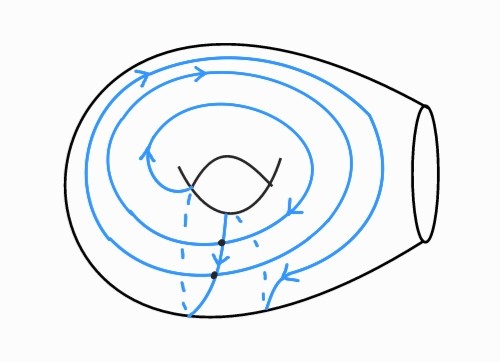}} 
    \caption{Non-simple closed curve $\eta^k$ on $S_{1,1}$.}
  \label{Representative of eta_2 eta_1} 
\end{figure}



\begin{figure}[h]
\centering
        \subfigure[Covering space of $S_{1,1}$ for $q=3$ and a simple lift of $ a^2b^3$ in this covering.]{{\includegraphics[width=6cm, height=5cm]{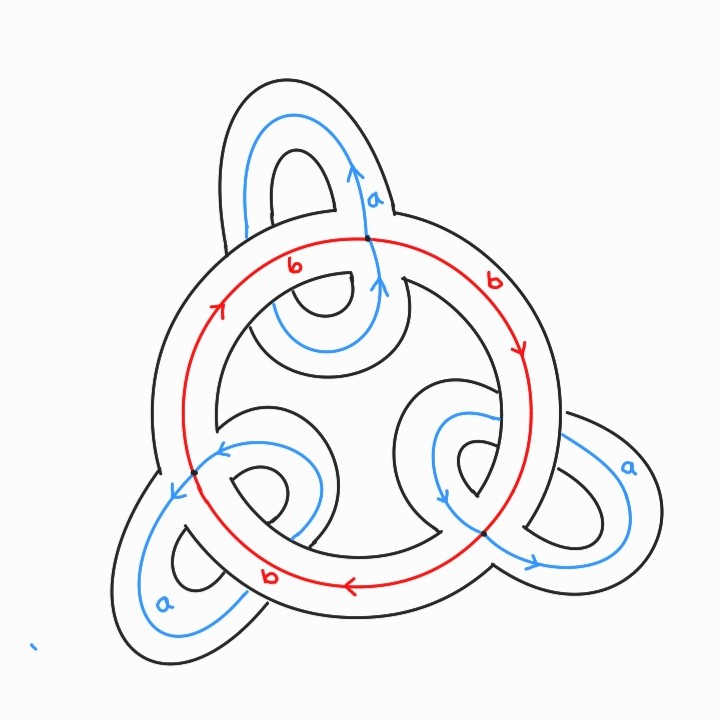}} \hspace{2cm} 
   {\includegraphics[width=6cm, height=5cm]{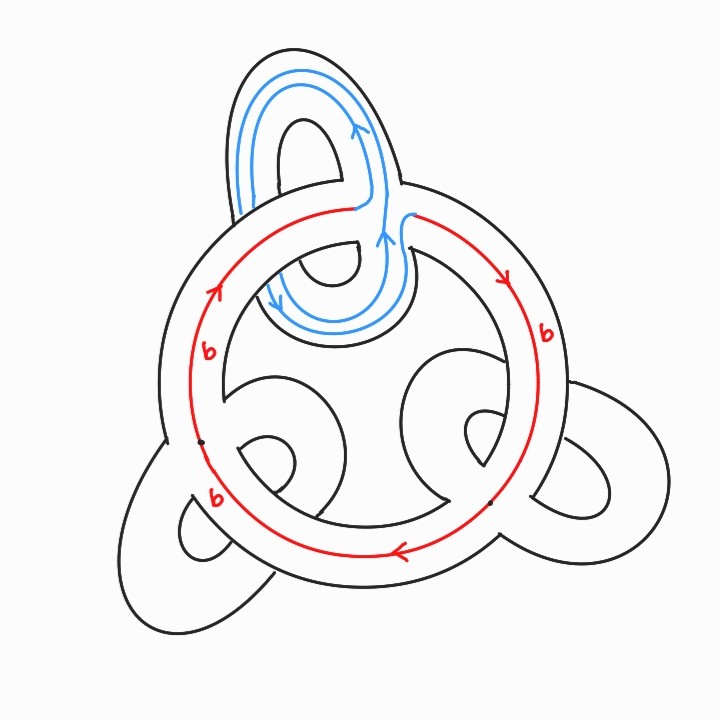}}}
 \hfill
 \end{figure}
 
 \begin{figure}[h]
 \centering
 \subfigure[Covering space of $S_{1,1}$ in case of $q=1$ and a simple lift of $a^2b^3$ in this covering.] 
{{\includegraphics[width=6cm, height=5cm]{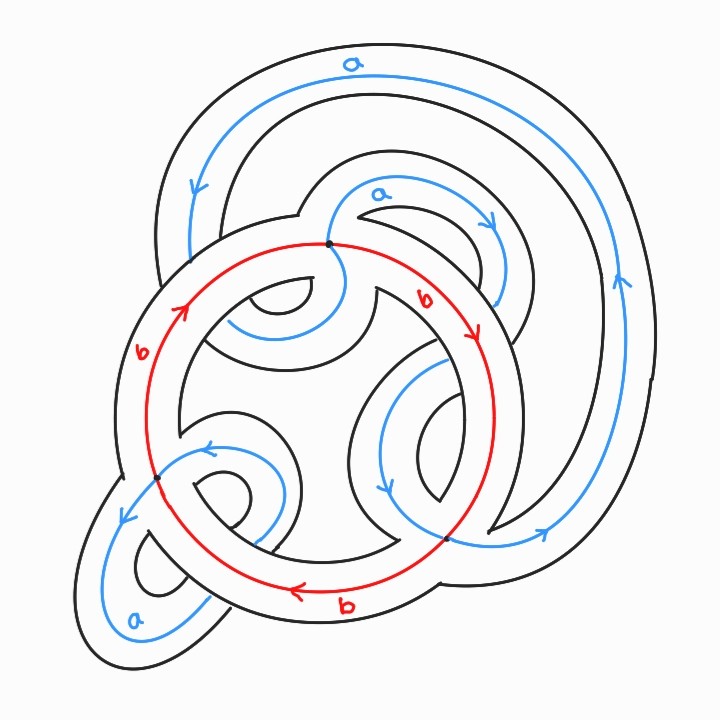}}\hspace{2cm} 
{\includegraphics[width=6cm, height=5cm]{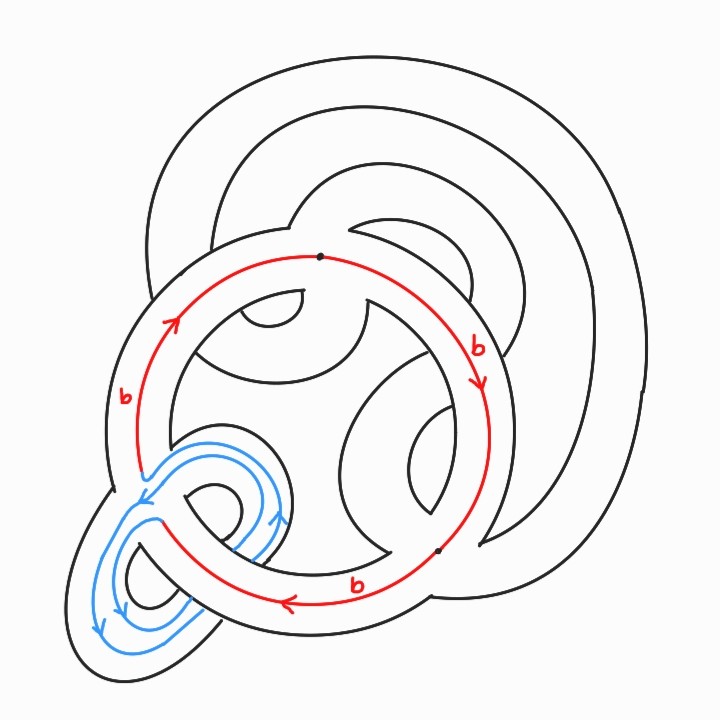}}} \hfill
\caption{}
\label{3 sheeted Covering space of S 1 1}
\end{figure}

\begin{remark}

There are curves with any number of self-intersections
that lift simply in coverings of degree 2 (see, \cite{MR4056691}). The above construction motivates us to give an explicit example of this fact. Let us denote the collection of closed curves $\sigma^n $ given by the homotopy class $a^nb^2$, then $\sigma^n$ is non-simple for all $n\geq 2$ as $i(\sigma^n,\sigma^n)=n-1$. Therefore for any self-intersection number $m$, we choose non-simple closed curve $\sigma^{m+1}$ on $S_{1,1}$. A lift of $\sigma^{m+1}$ is given by $m+1$ labeled $a$ edges followed by $2$ labeled $b$ edges (e.g., Figure \ref{2 sheeted covering space of S 1 1} for $m=2$). 
This gives an explicit example of the fact: $i(\alpha,\alpha) $ is not bounded by any function of $d(\alpha ) $ alone, where $\alpha$ is a free homotopy class of closed curve on an oriented surface $S$.
\end{remark}
\begin{figure}[H]
    \centering
   \subfigure[Two sheeted Covering space of $S_{1,1}$ (here $q=2$).]{\includegraphics[width=6cm, height=5cm]{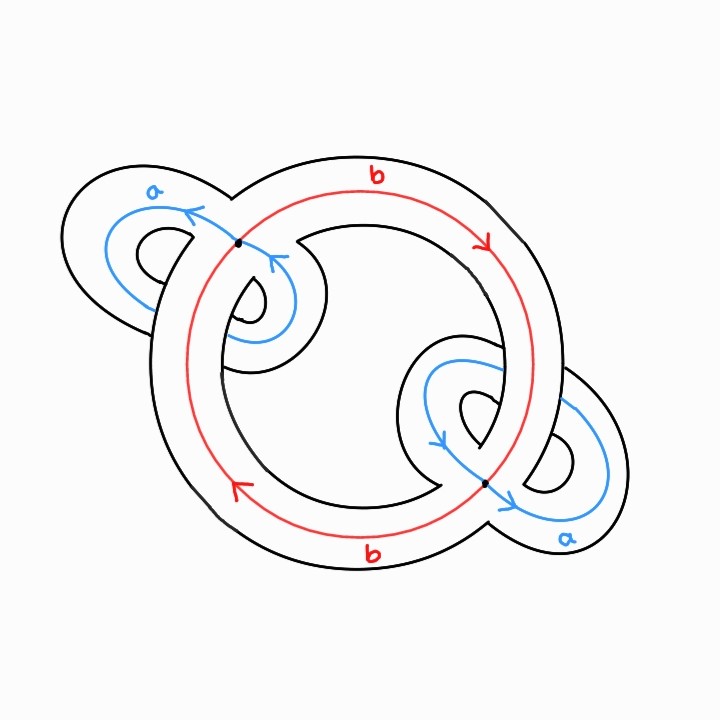 }} \hfill 
   \subfigure[Simple lift of curve $a^3b^2$ in this covering.]{\includegraphics[width=6cm, height=5cm]{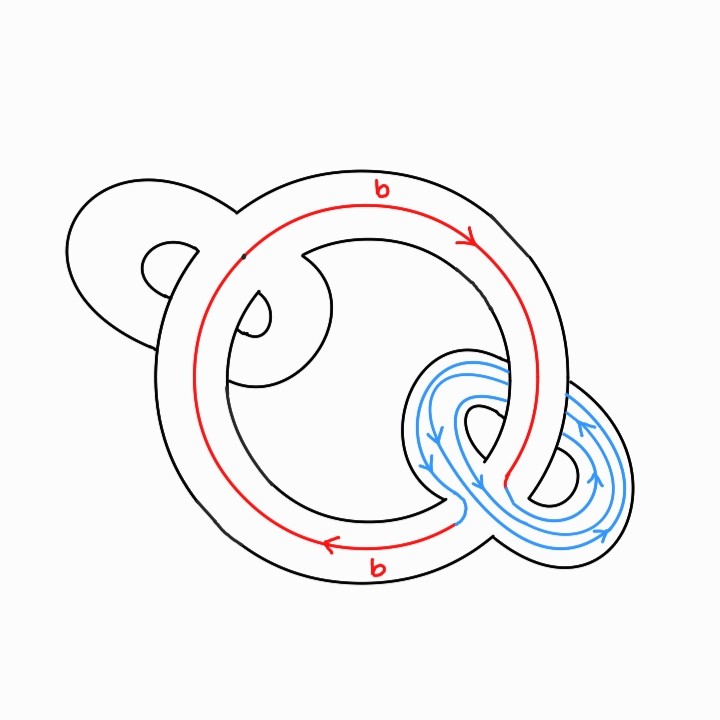}} 
    \caption{}
  \label{2 sheeted covering space of S 1 1} 
\end{figure}
\subsection{Surface of genus $g\geq 2$ and connected boundary}\label{Surface with genus greater equal to 2 and connected boundary}
Let $S_{g,1}$ be a surface with genus $g$ ($\geq 2$) and one boundary component. In this case, the genus $\tilde{g}$ of the $n$-sheeted covering of $S_{g,1}^{[n]}$ satisfies the following inequalities (see \cite[Section 3]{massey1974finite}):
\[ng-n+1 \leq \tilde{g} \leq n(g-1) + \frac{n+1}{2}.\]
Also for any integer $\tilde{g}$ satisfying this inequality, we follow the same construction as described in \cite[Section 3]{massey1974finite} to find $n$-sheeted covering space $S^{[n]}_{g,1}$: we write $S_{g,1}$ as boundary connected sum (see \cite{Algebric_topology_massey_MR211390} for definition of boundary connected sum) of a surface $M_1$ of genus $g-1$ and a surface $M_2$ of genus $1$, each having connected boundary (e.g. Figure \ref{Boundary connected sum of M_1 and M_2}). Let $\tilde{M_2}$ consist of $n$ copies of $M_2$ and $p_2: \tilde{M_2}\to M_2$ denote the map which maps each copy homeomorphically on $M_2$. Let, $p_1:\tilde{M_1}\to M_1$ be a $n$-sheeted covering of $M_1$. Then  a $n$-sheeted covering $p: S_{g,1}^{[n]}\to S_{g,1}$ is constructed by pasting $\tilde{M_1}$ and $\tilde{M_2}$ along $p_1^{-1}(M_1\cap M_2)$ and $p_2^{-1}(M_1\cap M_2)$ by appropriate homeomorphism. 
\begin{figure}[ht]
    \centering
     \includegraphics[width=10cm, height=3cm]{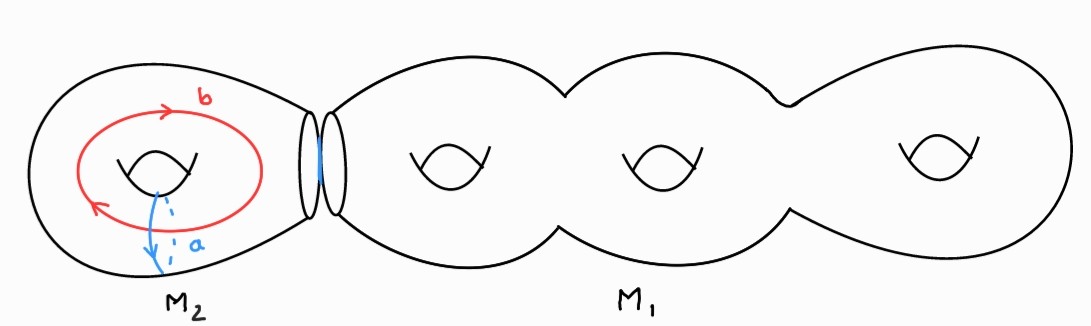}
    \caption{Decomposition of $S_{4,1}$ in boundary connected sum of $M_1$ and $M_2$. }
    \label{Boundary connected sum of M_1 and M_2}
\end{figure}

Now by the construction of the covering space, $S^{[n]}_{g,1}$ contains subsurface which is a $n$-sheeted covering $\tilde{M_1}$ of $M_1$. Let $M_{1}'$ and $M_{1}''$ are boundary connected sum decomposition of $M_1$ where $M_{1}'$ is of genus $g-2$ and $M_{2}''$ is of genus 1, each having connected boundary. Then $\tilde{M_1}$ contains subsurface which is $n$-sheeted covering of $M_{1}'$. Inductively after finitely many steps of boundary connected sum decomposition, $S^{[n]}_{g,1}$ contains a subsurface say $F$ which is $n$-sheeted covering of a surface of genus $1$ and connected boundary and identified with $S_{1,1}$. Let $p':F \to S_{1,1}$ be the corresponding covering map. Now using Section $\ref{Section Surface of genus g=1 and connected boundary}$ we can construct non-simple closed curves $\eta$ on $S_{1,1}$ which lift simply on that subsurface $F$ by $p'$. We note that these chosen curves $\eta$ are essential, so their lift. Therefore, no matter how glued as boundary connected sum to construct $p$, the curves $i \circ \eta $ lift simply by $p$ on $S_{g,1}$ where $i:S_{1,1}\to S_{g,1}$ is the inclusion map.

\subsection{Surface of genus $g\geq 1$ and boundary components $k=2$.}
Let $S_{g,2}$ be a surface with genus $g$ and $2$ boundary components. In this case, genus $\tilde{g}$ 
 of $n$-sheeted covering of $S_{g,2}$ satisfy the following relation (see \cite[Section 5]{massey1974finite}):
\[ng \geq \tilde{g}\geq ng-n+1.\]
\begin{figure}[ht]
    \centering
     \includegraphics[width=12cm, height=3.5cm]{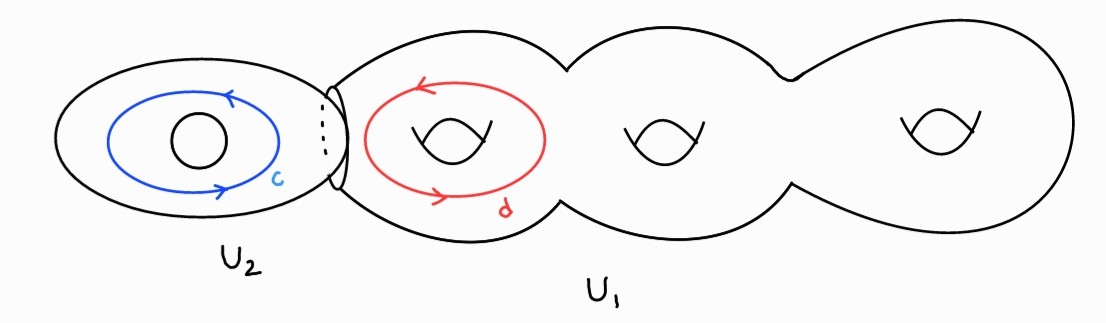}
    \caption{Decomposition of $S_{3,2}$ in boundary connected sum of $U_1$ and annulus $U_2$.}
    \label{Boundary connected sum of U_1 and U_2}
\end{figure}

Also for any given integer $\tilde{g}$ in this range we can construct covering using the method given in \cite[Section 5]{massey1974finite} which we recall below. Let $p:S_{g,2}^{[n]} \to S_{g,2}$ be a covering projection. We represent $S_{g,2}$ as the boundary connected sum of a surface $U_1$ of genus $g$ with connected boundary and an annulus $U_2$, i.e., a surface with $0$ genus and $2$ boundary components (e.g. Figure \ref{Boundary connected sum of U_1 and U_2}). Let $u=ng-\tilde{g}$, then the relation transformed to \[n-1\geq u\geq 0.\] Now $p^{-1}(U_1)$ have connected component as follows:
\begin{enumerate}
    \item[(1)]\namelabel{1} One component is an $(u+1)$-sheeted covering space of $U_1$ of the maximal possible genus $ug+g-u$.
    \item[(2)]\namelabel{2}  In case $n-1 > u$, the remaining components are each $1$-sheeted covering space of $U_1$, each of genus $g$. There are $n-(u+1)$ such components.
    \end{enumerate}
   Also, $p^{-1}(U_2)$ has connected components as follows:
\begin{enumerate}
  \item[(3)]\namelabel{3} One component is an $(n-u)$-sheeted covering space of $U_2$.
    \item[(4)]\namelabel{4} There are $u$ components, each of them is $1$-sheeted covering of $U_2$.
\end{enumerate}
Now we may put everything together to find a connected covering space of $S_{g,2}$ in a certain scheme: $S_{g,2}^{[n]}$ is a boundary connected sum of two pieces:
\begin{itemize}

    \item We attach the surface described in (\ref{1})  and $u$ annuli described in (\ref{4}) as boundary connected sum.
    \item We attach annulus described in (\ref{3}) with $n-u-1$ surfaces of genus $g$ described in (\ref{2}) along the boundary.
\end{itemize}

$\textit{Case 1:}$ Let, $n-1 \geq u \geq 1$. In this case covering spaces $S_{g,2}^{[n]}$ has a subsurface $Y$ which is $u+1 (\geq 2)$-sheeted covering space of $U_1$ by the construction of the covering space. we identify $U_1$ with $S_{g,1}$ and hence by Section \ref{Surface with genus greater equal to 2 and connected boundary}, for different values of $u$, we choose non-simple closed curves $\xi$ on $U_1$ which lifts simply on $Y$ and hence it is easy to show $i \circ \xi$ lifts simply on $S_{g,2}^{[n]}$ where $i: U_1\to S_{g,2} $ is inclusion map.

$\textit{Case 2:}$ Let, $u=0$. In this case, by the construction, we attach $n$-sheeted covering of $U_2$ and $(n-1)$ disjoint copies of $U_1$. Finally, attaching this surface with $U_1$ as boundary connected sum we get the required covering space. We choose the generator say, $c$ of $U_2$, and one generator $d$ of $U_1$ which traverses once around a genus (e.g. Figure \ref{Boundary connected sum of U_1 and U_2}). Without loss of generality, let $c$ and $d$ both be in the anticlockwise direction. Let, $\zeta^k$ be the closed curve on $S_{g,2}$ given by the homotopy class of $c^{k}d$ where $k\in \mathbb{N}$. Now we state the next lemma to ensure the chosen curves $\zeta ^k$ are non-simple for all $k\geq 2$.
\begin{figure}[ht]
    \centering
     \includegraphics[width=12cm, height=3.5cm]{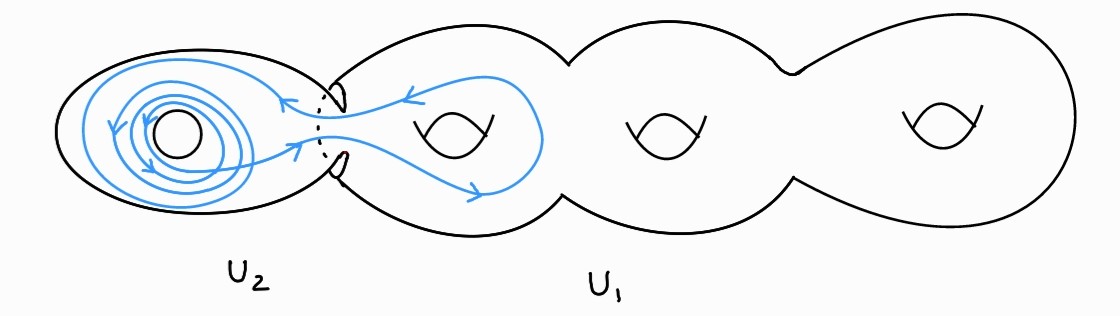}
    \caption{A minimal position representative of $\zeta^4 $ on $S_{3,2}$.}
    \label{minimal position representative of zeta 4}
\end{figure}
\begin{lemma}
    For $k \geq 2$, the curve $\zeta^k$ has $i(\zeta^k,\zeta^k)=k-1$.
\end{lemma}
\begin{proof}
    We can choose a representative of $\zeta^k$ which has $k-1$ self-intersection number (e.g. Figure \ref{minimal position representative of zeta 4}). Therefore, $i(\zeta^k, \zeta^k)\leq {k-1}$. Now using bigon criteria we can show that $i(\zeta^k, \zeta^k)=k-1$.
\end{proof} 

For $n$-sheeted covering $S_{g,2}^{[n]}$ we choose non-simple closed curves $\zeta ^n$ on $S_{g,2}$. We recall standard $n$-sheeted covering of annulus $U_2$ and lift of the primitive curve $c^n$ which is simple. Let $p'$ be the point where $c$ and $d$ concatenate and $p_1,p_2,\ldots,p_n$  on $S_{g,2}^{[n]}$ are preimages of $p'$. In this case lift $\zeta^n \equiv c^nd$ homotopic to a lift of $c^n$ joined with $n$ copies of $d$ at $p_1,p_2\ldots, p_n$ on $S_{g,2}^{[n]}$. Consequently, $\zeta^n$ lifts simply on $S_{g,2}^{[n]}$.
For example, in case of $3$-sheeted covering of $S_{3,2}$ we choose non-simple closed curve $\zeta^3$ given by Figure \ref{non-simple closed curve zeta^3} and its simple lift on $S_{3,2}^{[n]}$ is depicted in Figure \ref{Simple lift in case g=3,k=2}.
\begin{figure}[H]
    \centering
     \includegraphics[width=12cm, height=3.5cm]{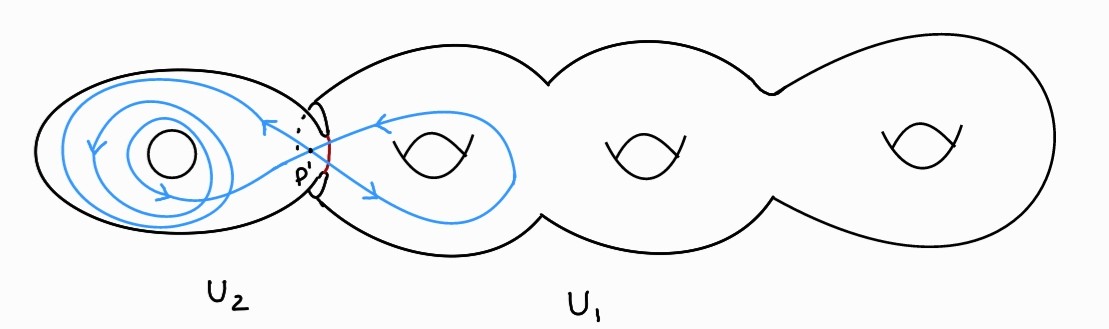}
    \caption{Non-simple closed curve $\zeta^3$ on $S_{3,2}$.}
    \label{non-simple closed curve zeta^3}
\end{figure}

\begin{figure}[H]
    \centering
     \includegraphics[width=12cm, height=8cm]{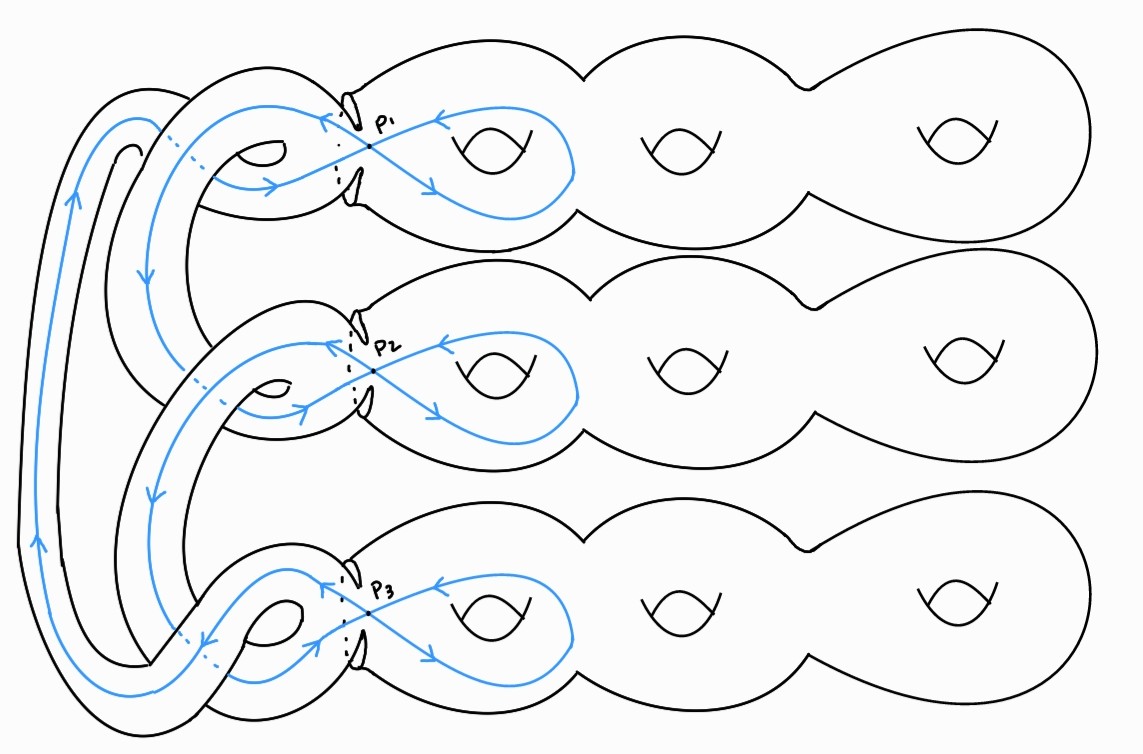}
    \caption{Simple lift of $\zeta^3$ on $S_{3,2}^{[3]}$.}
    \label{Simple lift in case g=3,k=2}
\end{figure}
\subsection{Surface of genus $g\geq 1$ and number of boundary components $k\geq 3$.}\label{Surface of genus greater equal to 1  boundary components greater equal to 3}
Let $S_{g,k}$ be surface having genus $g\geq 1$ and boundary components $k\geq 3$. Any $n$-sheeted covering of $S_{g,k}$ satisfies the inequalities (see \cite[Section 6]{massey1974finite}),
\[ng+\frac{1}{2}(n-1)k-n+1\geq \tilde{g}\geq ng-n+1.\]
Also for any integer $\tilde{g}$, satisfying this condition we can construct covering space having genus $\tilde{g}$, following method in \cite[Section 6]{massey1974finite} which we recall below. Let $p: S_{g,k}^{[n]} \to S_{g,k}$ be a $n$-sheeted covering. We write $S_{g,k}$ as boundary connected sum of $N_1$ of genus $g$ with connected boundary and a surface $N_2$ of genus $0$ and $k$ boundary components (e.g. Figure \ref{Boundary connected sum of N1 and N2.}). We consider three different cases to construct non-simple closed curves on $S_{g,k}$ for different covering spaces.
\begin{figure}[H]
    \centering
     \includegraphics[width=12cm, height=3.5cm]{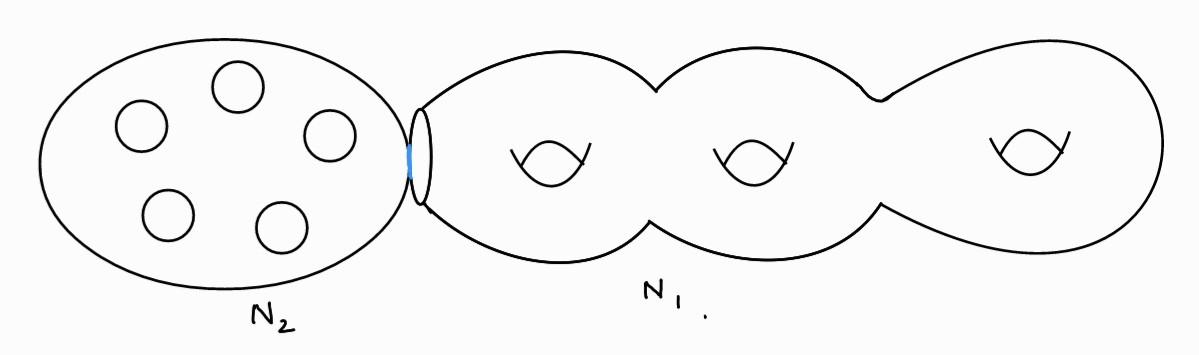}
    \caption{Decomposition of $S_{3,6}$ in boundary connected sum of $N_1$ and $N_2$.}
    \label{Boundary connected sum of N1 and N2.}
\end{figure}

\textit{Case 1:}\label{Case 1 S_g,k} Let, $ng-\frac{1}{2}(n-1)\geq \tilde{g}\geq ng-n+1$. By using Section \ref{Section Surface of genus g=1 and connected boundary} and Section \ref{Surface with genus greater equal to 2 and connected boundary}, we construct $n$-sheeted covering space of $N_1$, $p_1:N_1^{[n]}\to N_1$ such that $N_1^{[n]}$ has genus $\tilde{g}$. Then we construct $S_{g,k}^{[n]}$ as boundary connected sum of $N_1^{[n]}$ and $n$ copies of $N_2$. Again by using Section \ref{Section Surface of genus g=1 and connected boundary} and Section \ref{Surface with genus greater equal to 2 and connected boundary} we can construct non-simple closed curve $\eta$ on $N_1$ which lifts simply on $N_1^{[n]}$ by $p_1$. Due to the construction of the covering, we observe that $p|_{N_1^{[n]}-{B}}=p_1|_{N_1^{[n]}-{B}}$ where $B$ be set of boundary components of $N_1^{[n]}$. Also the non-simple closed curves $\eta$ are essential, hence $\tilde{\eta}$ by Lemma \ref{Essential closed curve}. Therefore the image of $\tilde{\eta}$ does not depend on how the boundary of $N_1^{[n]}$ is glued with $n$ copies of $N_2$. Consequently, $i \circ \eta$ lift simply on $S_{g,k}^{[n]}$ by $p$ where $i:N_1 \to S_{g,k} $ is inclusion map.

\textit{Case 2:} Let, $ng+\frac{1}{2}(n-1)k-n+1 \geq \tilde{g}\geq ng$. Using Section \ref{Surface of genus g=0 and boundary components k greater equal to 3} we construct $n$-sheeted covering $p_2: N_2^{[n]}\to N_2$ such that $N_2^{[n]}$ has the genus $\tilde{g}-ng$. Then we construct $S_{g,k}^{[n]}$ as boundary connected sum of $N_2^{[n]}$ and $n$ copies of $N_1$. By using Section \ref{Surface of genus g=0 and boundary components k greater equal to 3} we can construct non-simple closed curves $\gamma$ on $N_2$ which lifts simply on $N_2^{[n]}$. Again by similar argument as previous case, $i\circ \gamma$ lifts simply on $S_{g,k}^{[n]}$ by $p$ where $i:N_2 \to S_{g,k} $ is inclusion map.

\textit{Case 3:} Let, $ng >\tilde{g}>ng-\frac{1}{2}(n-1)$. If we denote $l=ng-\tilde{g}+1$, the condition simplifies to $1<l<\frac{n+1}{2}$. Now we construct $p: S_{g,k}^{[n]} \to S_{g,k}$ such that $p^{-1}(N_1)$ consists of $n-l+1$ components as follows:
\begin{itemize}
    \item The $(n-l)$ components of $p^{-1}(N_1)$ are each $1$-sheeted covering of $N_1$.
    \item The remaining $1$ component of $p^{-1}(N_1)$ is an $l$-sheeted covering of $N_1$ of minimum possible genus , i.e., genus $l(g-1)+1$. 
\end{itemize}
Also, we construct $p: S_{g,k}^{[n]} \to S_{g,k}$ such that each component of $p^{-1}(N_2)$ has genus $0$. Now $S_{g,k}^{[n]}$ is described as boundary connected sum of two pieces as follows:
\begin{enumerate}
    \item The first piece consists of $l$-sheeted covering of $N_1$ (of genus $l(g-1)+1$ ) to which $l-1$ copies of $N_2$ will be attached as boundary connected sum.
    \item The second piece consists of $(n-l+1)$-sheeted covering space of $N_2$ of genus $0$ to which $(n-l)$ copies of $N_1$ has been attached as boundary connected sum.
\end{enumerate}
Hence, in this case, $S_{g,k}^{[n]}$ contains a subsurface which contains an $l$-sheeted covering of $N_1$ where $l>1$ and we described how to construct non-simple closed curves on $N_1$ which lifts simply in Section \ref{Surface with genus greater equal to 2 and connected boundary} (for genus of $N_1$ greater equal to $2$) and Section \ref{Section Surface of genus g=1 and connected boundary} (for genus of $N_1$ equal to $1$). By similar arguments as \textit{Case 1}(Section \ref{Case 1 S_g,k}), these curves lift simply on $S_{g,k}^{[n]}$ by $p$.
\section{Closed Surfaces}\label{Closed orientable surfaces}
In this section, we prove Theorem \ref{Main theorem} for closed surfaces.
We start with the following lemma.
\begin{lemma}\label{Theorem 1}
    Assume $P:\tilde{X} \to X$ is a covering map with $\tilde{X}, X$ both paths connected. Assume $A$ is path connected subset of $X$ so that $i_*: \pi_1(A,a)\to \pi_1(X,a)$ is onto for any $a\in A$ where $i$ is an inclusion map. Then $P^{-1}(A)$ is path connected.
\end{lemma}
\begin{proof} Let $p_1,p_2$ be two points in $P^{-1}(A)$, so, $P(p_1)$ and $P(p_2)$ are in $A$. Since, $A$ is path connected, there is a path say $g:[0,1]\to A$ such that $g(0)=P(p_1)$ and $g(1)=P(p_2)$. Using path lifting property, $\tilde{g}$ is lift of $g$ starting at $p_1$ and ending at some point in the fiber $P^{-1}(g(1))$, say $p_3$. Again, path connectedness of $\tilde{X}$  implies that there is a path between $p_2$ and $p_3$ say $f$. Therefore,, $P\circ f$ is a loop based at $P(p_2)(=P(p_3))$. Consequently, $P \circ f \in \pi_1(X,P(p_2))$. The surjectivity of $i_*$ shows there exists some $h \in \pi_1(A, P(p_3))$ such that $i \circ h$ is homotopic to $P\circ f $. Using the homotopy lifting property, $\widetilde{i \circ h}$ is a path in $\tilde{X}$ that begins at $p_3$ and ends at $p_2$. Hence $\tilde{h}$ is a path in $P^{-1}(A)$ starting at $p_3$ and ends at $p_2$. Consequently, $\tilde{g}\cdot \tilde{h}$ is a required path in $P^{-1}(A)$ joining $p_1$ and $p_2$. 
\end{proof}

\begin{proof}[\textbf{Proof of Theorem \ref{Main theorem} for closed surface $S_g$, $g\geq 2$}]
    We recall any covering of an oriented surface is oriented and any finite sheeted covering of a compact surface is again compact. Also, we know that two oriented closed surfaces $X$ and $Z$ are homeomorphic if and only if their Euler characteristics are equal. Now using Euler characteristics arguments, for a given $n$-sheeted covering $S_g^{[n]}$ of $S_g$ is homeomorphic to $S_{1+ng-n}$, i.e, they are topologically equivalent.
The fundamental group of $S_g$ is given by, \[\pi(S_g)= \big<c_1,d_1,c_2,d_2,\ldots c_g,d_g: [c_1,d_1][c_2,d_2]\ldots[c_g,d_g]=1\big>.\]

\begin{figure}[ht]
    \centering
  \includegraphics[width=6cm, height=6cm]{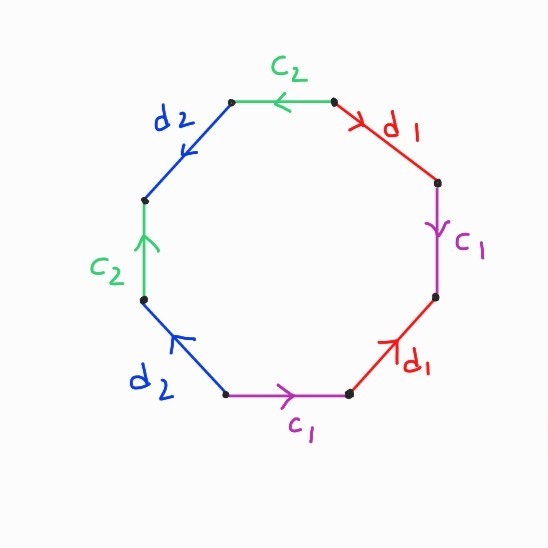}\hfill
   \includegraphics[width=6cm, height=6cm]{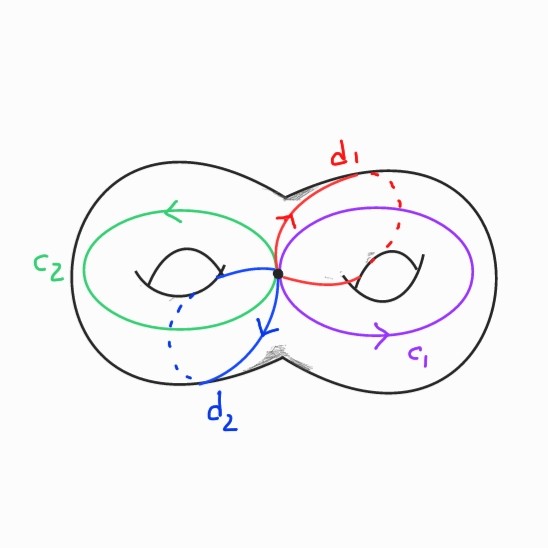}
    \caption{A standard set of generators of $S_2$.}
    \label{A standard set of generators of S2}
\end{figure}
\begin{figure}[ht]
    \centering
   \includegraphics[width=6cm, height=4cm]{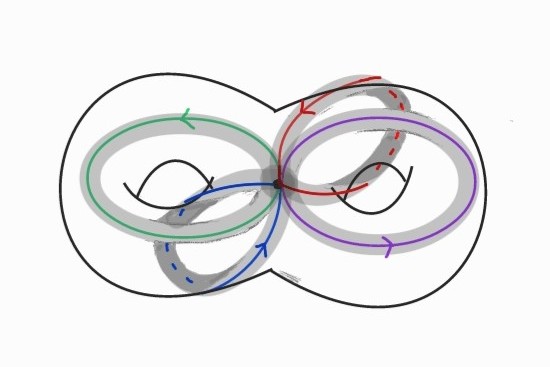}\hfill
  \includegraphics[width=6cm, height=4cm]{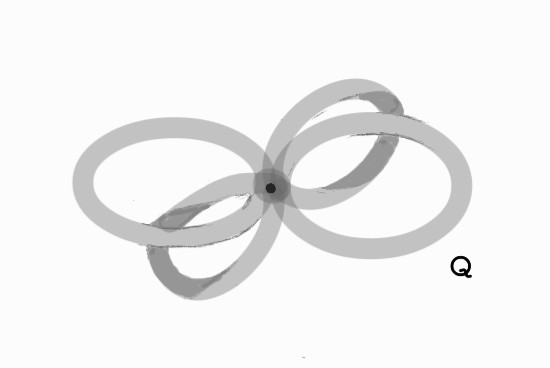} 
    \caption{Regular neighbourhood of generators of $S_2$.}
    \label{Regular neighbourhood of generators}
\end{figure}
Let us consider a regular neighborhood of the union of the generators $c_1, d_1, c_2, d_2, \ldots,$ $ c_g,d_g$ on $S_g$, which is a compact surface of genus $g$ and connected boundary, namely $Q$. For $g=2$ standard set of generators and this regular neighbourhood is shown in Figure \ref{Regular neighbourhood of generators} and Figure \ref{A standard set of generators of S2}. The fundamental group of $Q$ is a free group $F_{2g}$ with $2g$ generators. Without loss of generality, the generators of $Q$ are considered along generators of $S_{g}$. 
Let $P_n: S_{1+ng-n}\rightarrow S_g$ is any $n$-sheeted covering. Also, the induced homomorphism $i_*: \pi(Q)\to \pi(S_g)$ is onto where $i$ is an inclusion map from $Q \to S_g$. Then using Lemma \ref{Theorem 1}, $P_n^{-1}(Q)$ is path connected, therefore $P_n\vert_{P_n^{-1}(Q)}: P_n^{-1}(Q)\to Q $ is a $n$-sheeted covering of $Q$. We opt Section \ref{Surface with genus greater equal to 2 and connected boundary} to choose suitable non-simple closed curves on $Q$ which lifts simply on $P_n^{-1}(Q)$ by $P_n\vert_{P_n^{-1}(Q)}$. Consequently, these chosen non-simple closed curves lift simply on $S_{1+ng-n}$ by $P_n$. This completes the proof.
\end{proof}

\bibliographystyle{plain}
\bibliography{bibfile}

    
    



\end{document}